\documentclass[a4paper, 10pt, twoside, notitlepage]{amsart}

\usepackage[utf8]{inputenc}
\usepackage{color}
\usepackage{amsmath} 
\usepackage{amssymb} 
\usepackage{amsthm}
\usepackage{geometry}
\usepackage{graphicx}
\usepackage{esint}
\usepackage[ocgcolorlinks,linkcolor=blue]{hyperref}

\theoremstyle{plain}
\newtheorem{thm}{Theorem}
\newtheorem{prop}{Proposition}[section]
\newtheorem{lem}[prop]{Lemma}
\newtheorem{cor}[prop]{Corollary}

\newtheorem{defi}[prop]{Definition}
\newtheorem{rmk}[prop]{Remark}

\newcommand {\R} {\mathbb{R}} \newcommand {\Z} {\mathbb{Z}}
 \newcommand {\N} {\mathbb{N}}
 
\newcommand {\p} {\partial}

\newcommand {\D} {\Delta}

\newcommand {\sgn} {\text{sgn}}

\newcommand{\q}{(n+2)}
\newcommand{\ql}{n+2}
\newcommand{\qc}{2n+3}
\newcommand{\mat}{a_{m_1,k_1,l_1}^{m_2,k_2,l_2}}
\newcommand{\matb}{b_{m_1,k_1,l_1}^{m_2,k_2,l_2}}

\DeclareMathOperator{\argmin}{argmin}

\DeclareMathOperator{\spa} {span}

\DeclareMathOperator{\F} {\mathcal{F}}

\title{Exponential instability in the fractional Calder\'on problem}

\author[A. R\"uland]{Angkana R\"uland}
\address{Max-Planck-Institute for Mathematics in the Sciences, Inselstraße 22,
04103 Leipzig,
Germany}
\email{rueland@mis.mpg.de}

\author[M. Salo]{Mikko Salo}
\address{University of Jyvaskyla, Department of Mathematics and Statistics, PO Box 35, 40014 University of Jyvaskyla, Finland}
\email{mikko.j.salo@jyu.fi}

\begin{document}

\maketitle

\begin{abstract}
In this note we prove the exponential instability of the fractional Calder\'on problem and thus prove the optimality of the logarithmic stability estimate from \cite{RS17}. In order to infer this result, we follow the strategy introduced by Mandache in \cite{M01} for the standard Calder\'on problem. Here we exploit a close relation between the fractional Calder\'on problem and the classical Poisson operator. Moreover, using the construction of a suitable orthonormal basis, we also prove (almost) optimality of the Runge approximation result for the fractional Laplacian, which was derived in \cite{RS17}.
Finally, in one dimension, we show a close relation between the fractional Calder\'on problem and the truncated Hilbert transform.
\end{abstract}

\section{Introduction}

In this note we consider the fractional Calder\'on problem, which was introduced in \cite{GSU16}. More precisely, we consider the mapping
\begin{equation}
\label{eq:radial}
\begin{split}
\Lambda_q: \widetilde{H}^{s}(B_3 \setminus \overline{B}_2) &\rightarrow H^{-s}(B_3 \setminus \overline{B}_2),\\
f& \mapsto (-\D)^s u|_{B_3 \setminus \overline{B}_2},
\end{split}
\end{equation}
where $u$ is a solution to 
\begin{equation}
\label{eq:problem}
\begin{split}
((-\D)^s + q) u & = 0 \mbox{ in } B_1,\\
u &= f \mbox{ in } \R^n \setminus B_1,
\end{split}
\end{equation}
with $q\in L^{\infty}(B_1)$ and $n\geq 1$. Here we denote by $H^s$ the standard $L^2$ based Sobolev spaces, and $\widetilde{H}^s(U)$ is the closure of $C^{\infty}_c(U)$ in $H^s(\R^n)$.

We are interested in the (in)stability properties of the inverse problem of determining $q$ from the knowledge of $\Lambda_q$, with $\Lambda_q$ as in \eqref{eq:radial}.
Due to the results from \cite{GSU16}, it is known that, if the potentials $q_1, q_2$ are such that $(-\D)^s + q_i$ does not have zero as an eigenvalue, then
\begin{align*}
\Lambda_{q_1} = \Lambda_{q_2} \ \implies \ q_1 = q_2.
\end{align*}
This has been quantified in \cite{RS17}, where a logarithmic stability estimate was proved. It is the purpose of this note to show that this logarithmic estimate is optimal and that the problem indeed has an exponential instability. Our main result in this context can be formulated as follows:

\begin{thm}
\label{thm:main}
Let $B_1 \subset \R^n$ with $n\geq 1$.
For any $m\in \N$ there exist constants $r_0>0$ and $\beta>0$ such that for all sufficiently small $\epsilon>0$ and for all $q_0 \in L^{\infty}(B_1)$ with $\|q_0\|_{L^{\infty}(B_1)}\leq \frac{r_0}{2}$ there are potentials $q_1,q_2$ with $q_i - q_0 \in C^m_0(B_1)$ such that
\begin{equation}
\label{eq:instab}
\begin{split}
\|\Lambda_{q_1}-\Lambda_{q_2}\|_{L^2(B_3 \setminus \overline{B}_2) \to L^2(B_3 \setminus \overline{B}_2)}& \leq 8 \exp(-\epsilon^{-\frac{n}{(2n+3)m}}),\\
\|q_1-q_2\|_{L^{\infty}(B_1)} & = \epsilon,\\
\|q_i-q_0\|_{C^m(\overline{B}_1)} & \leq \beta,\\
\|q_i-q_0\|_{L^{\infty}(B_1)} & \leq \epsilon, \quad i\in\{1,2\}.
\end{split}
\end{equation}
\end{thm}

In contrast to the instability examples for the classical Calder\'on problem, we de not need $q_0$ to be compactly supported in $B_1$.

\begin{rmk}
\label{rmk:norms}
We note that although $\widetilde{H}^s(B_3 \setminus \overline{B}_2)$ and $H^{-s}(B_3 \setminus \overline{B}_2)$ are the natural function spaces in the context of the fractional Calder\'on problem (cf.\ the well-posedness discussion in Lemma 2.3 in \cite{GSU16}), the disjointness of the domains $B_1$ and $B_3 \setminus \overline{B}_2$ allows us to also define the exterior problem \eqref{eq:problem} for $L^2(B_3 \setminus \overline{B}_2)$, see Remarks \ref{rmk:well} and \ref{rmk:well1} for more details on this. For simplicity we state our results in terms of these $L^2$ norms instead of the norms used in \cite{RS17}.
\end{rmk}

As a corollary of Theorem \ref{thm:main}, we directly obtain the exponential instability of the fractional Calder\'on problem in a neighbourhood of a suitable potential $q$:

\begin{cor}
\label{cor:main}
Let $n\geq 1$ and $B_1 \subset \R^n$, assume that $m\in \N$. There is a potential $q\in L^{\infty}(B_1)$, a sequence of errors $\{\epsilon_k\}_{k\in \N}$ with $\epsilon_k \rightarrow 0$ and a sequence of potentials $\{q_k\}_{k\in \N} \subset C^{m}(\overline{B}_1)$ such that 
\begin{itemize}
\item[(i)] $\|q-q_k\|_{L^{\infty}(B_1)} \geq \frac{\epsilon_k}{2}$ and
$\|\Lambda_{q}-\Lambda_{q_k}\|_{L^2(B_3 \setminus \overline{B}_2) \rightarrow L^2(B_3 \setminus \overline{B}_2)} \leq C \exp(-\epsilon^{-\frac{n}{(2n+3)m}}_k)$,
\item[(ii)] for all $0\leq m'\leq m$ we have $q_k \rightarrow q$ in $C^{m'}(\overline{B}_1)$.
\end{itemize}
\end{cor}

Both of these results rely on a combination of ideas introduced by Mandache in the context of the classical Calder\'{o}n problem, and a close relation between the fractional Calder\'on problem and the Poisson problem for the Laplacian.

Let us explain this in more detail: In the article \cite{M01} Mandache showed that the Calder\'on problem has an exponential instability, and that thus the stability results of Alessandrini (see \cite{A88}, \cite{A90}) are optimal. To achieve this, Mandache used two main ingredients in his proof, which were later also exploited by Di Cristo and Rondi \cite{DR03}, who systematically derived similar estimates to prove exponential instability in a variety of inverse problems. 
\begin{itemize}
\item Firstly, he deduced very strong smoothing properties for the local analogue of the operator $\Gamma(q):=\Lambda_{q}-\Lambda_0$.
\item Secondly, he made use of the ``separating" properties of the space $L^{\infty}$. These are to be understood in the sense of Lemma \ref{lem:separate} and Definition \ref{defi:Banach}.  
\end{itemize}
The latter property is a general property of the space $L^{\infty}$ and hence does not depend on the problem at hand. As a consequence, we can also rely on this in the context of the fractional Calder\'on problem. It is the first property, i.e.\ the \emph{smoothing} properties of $\Gamma(q)$, which is based on the specific problem. It is thus one of the main contributions of this note to show that the strong smoothing properties still hold for the fractional Calder\'on problem. Here, as had been observed in \cite{DR03}, the key step consists of the construction of a suitable orthonormal basis (Lemma \ref{lem:orthog}), which interacts well with the fractional Laplacian.

In order to deduce the regularizing properties of $\Gamma(q)$, we rely on an identity that relates the Poisson operator for the fractional Laplacian to the Poisson operator for the Laplacian (see Lemma \ref{lem:orthog}).

As a further consequence of the construction of the orthonormal basis from Lemma \ref{lem:orthog}, we deduce the (almost) optimality of the Runge approximation property for the fractional Laplacian, which was given in Theorem 1.3 in \cite{RS17} and which quantified the approximation results from \cite{DSV14} and \cite{GSU16}:

\begin{thm}
\label{prop:approx}
There is a sequence $\{v_{p}\}_{p\in\N} \subset L^2(B_1)$ with the following property: If $u \in H^s(\R^n)$ satisfies
\begin{itemize}
\item[(i)]  
$
(-\D)^s u = 0 \mbox{ in } B_1, \quad
$
$u|_{\R^n \setminus \overline{B}_1} = f$ for some $f\in L^2(B_3 \setminus \overline{B}_2)$,
\item[(ii)] $\|u - v_{p}\|_{L^2(B_1)} \leq p^{-1}$,
\end{itemize}
we have that for some constant $c_0$, which only depends on $n,s$,
\begin{align*}
\|f\|_{L^2(B_3 \setminus \overline{B}_2)} \geq c_0 2^{p} \|v_{p}\|_{L^2(B_1)}.
\end{align*}
\end{thm}

\begin{rmk}
\label{rmk:almost}
We refer to Theorem \ref{prop:approx} as \emph{almost} showing optimality in our approximation result from \cite{RS17}, as on the one hand the result of Theorem \ref{thm:exp_inst_approx} gives evidence of the exponential dependence of the control on the admissible error. But on the other hand, this exponential behaviour is formulated with respect to $L^2$ norms only, instead of using the combination of $L^2$ and $H^{s}$ norms as in the statement of Theorem 2 in \cite{RS17}.
\end{rmk}

Finally, we remark that in one dimension a very precise relation between the fractional Laplacian and the truncated Hilbert transform can be obtained (see Lemma \ref{lem:HT_frac} in Section \ref{sec:1D}). This in particular allows us to exploit the well-studied properties of the truncated Hilbert transform and its singular value decomposition \cite{K10}, \cite{KT12} in order to prove exponential instability results in one-dimension. We illustrate this by presenting a second variant of Theorem \ref{prop:approx} and its proof in Section \ref{sec:1D} (cf.\ Theorem \ref{thm:exp_inst_approx}).

The remainder of the article is organized as follows: In Section \ref{sec:smoothing} we present the main novelty of the article and deduce the central smoothing estimates for the fractional Calder\'on problem. This is achieved by constructing a suitable orthonormal basis adapted to the fractional Laplacian (cf.\ Lemma \ref{lem:orthog}).
In Section \ref{sec:metric} we combine the smoothing estimate with the general properties of $L^{\infty}$. Here we follow the analogous reasoning of Mandache \cite{M01}. The combination of the results from Sections \ref{sec:smoothing} and \ref{sec:metric} then provide the proofs of Theorem \ref{thm:main} and of Corollary \ref{cor:main}. Using the orthonormal basis from Section \ref{sec:smoothing} we present the argument for Theorem \ref{prop:approx} in Section \ref{sec:approx}. In order to illustrate the explicit connection between the truncated Hilbert transform (and its well-studied singular value basis), we discuss the one-dimensional set-up in more detail in Section \ref{sec:1D}. Exploiting the singular value decomposition of the truncated Hilbert transform, we there derive another variant of Theorem \ref{prop:approx}. Finally, in Section \ref{sec:Hadamard} we show that the classical Hadamard example of exponential instability also holds in the context of the Caffarelli-Silvestre extension with $s\in(0,1)$.

\subsection*{Acknowledgements}
M.S.\ is supported by the Academy of Finland (Finnish Centre of Excellence in Inverse Problems Research, grant numbers 284715 and 309963) and an ERC Starting Grant (grant number 307023).

\section{Smoothing}
\label{sec:smoothing}

In this section we construct an orthonormal basis $\{f_{m,k,l}\}$ of $L^2(B_3 \setminus \overline{B}_2)$ such that the corresponding solutions $u_{m,k,l}$ to the fractional exterior value problem
\begin{equation}
\label{eq:Poisson}
\begin{split}
(-\D)^s u_{m,k,l} & = 0 \mbox{ in } B_1,\\
u_{m,k,l} & = \chi_{B_{3}\setminus \overline{B}_2} f_{m,k,l} \mbox{ in } \R^n \setminus \overline{B}_1,
\end{split}
\end{equation}
satisfy exponential decay bounds. We will then translate this decay into smoothing properties for the operator $\Gamma(q):= \Lambda_{q}-\Lambda_0$ (c.f.\ Proposition \ref{prop:smooth}), which will play a central role in our instability argument.

\begin{lem}
\label{lem:orthog}
Let $n\geq 1$ and $B_1 \subset \R^n$. For $m\in \N$ denote by $\{h_{m,l}\}_{0 \leq l \leq l_m}$ the spherical harmonics of degree $m$ on $\partial B_1$. 
Let $s\in(0,1)$ and consider the mapping $A_0$ defined by
\begin{align*}
L^2(B_3 \setminus \overline{B}_2) \ni f \mapsto r_{B_1} u \in L^{2}(B_1),
\end{align*}
where $u$ is the solution of $(-\Delta)^s u = 0$ in $B_1$ with $u = \chi_{B_{3}\setminus \overline{B}_2} f$ in $\R^n \setminus \overline{B}_1$, and $r_{B_1}$ denotes the restriction onto the unit ball.
Then there exists an orthonormal basis $\{ f_{m,k,l} \}_{m, k \in \N, 0 \leq l \leq l_m}$ of $L^2(B_3 \setminus \overline{B}_2)$, where 
\[
f_{m,k,l}(x) = g_{m,k}(|x|)h_{m,l}(x/|x|),
\]
such that for constants $C,c>0$ which only depend on $n,s$,
\begin{align*}
\|A_0 f_{m,k,l}\|_{L^2(B_1)} \leq C e^{-c(m+k)}.
\end{align*}
\end{lem}

\begin{rmk}[Well-posedness with exterior $L^2(B_3 \setminus \overline{B}_2)$ data]
\label{rmk:well}
As already pointed out in Remark \ref{rmk:norms}, a natural set-up for the exterior problem \eqref{eq:Poisson} for the fractional Laplacian is to consider exterior Dirichlet data in $\widetilde{H}^{s}(B_3 \setminus \overline{B}_2)$ (see Lemma 2.3 in \cite{GSU16}). If, as in Lemma \ref{lem:orthog} and in various other places of the article, the exterior data are however localized onto a subset which is disjoint from the domain in which the fractional Schrödinger equation is considered, then it is also possible to deduce well-posedness of the fractional Schrödinger equation with less regular exterior data.

Let us explain this in the present context of data in $L^2(B_3 \setminus \overline{B}_2)$ with a fractional Schrödinger equation posed on $B_1$ (as in Lemma \ref{lem:orthog}). If $f \in L^2(B_3 \setminus \overline{B}_2)$, we may find $u \in H^s(\R^n)$ solving $((-\Delta)^s + q)u = 0$ in $B_1$ with $u|_{\R^n \setminus \overline{B}_1} = f$ by writing $u = E_0 f + v$, where $E_0$ denotes extension by zero and $v$ solves $((-\Delta)^s + q)v = F$ in $B_1$ with $v|_{\R^n \setminus \overline{B}_1} = 0$, and $F = -(-\Delta)^s (E_0 f)|_{B_1}$. But since $f$ is supported in $\overline{B}_3 \setminus B_2$, one has $F \in C^{\infty}(\overline{B}_1)$ by the pseudolocal property of Fourier multipliers, and one can find a solution $v \in H^s(\R^n)$ by standard well-posedness theory.

The previous argument shows well-posedness with exterior Dirichlet data $f$ in $L^2(B_3 \setminus \overline{B}_2)$, and also the estimate $\|u\|_{L^2(\R^n)} \leq \|f\|_{L^2(B_3 \setminus \overline{B}_2)} + C \|F\|_{H^{-s}(B_1)} \leq C \|f\|_{L^2(B_3 \setminus \overline{B}_2)}$ which again follows from the pseudolocal property.
\end{rmk}

\begin{proof}
The proof relies on the explicitly known Poisson formula for the fractional Laplacian in $B_1$ (see for instance \cite{B15}) and its close relation to the Poisson formula for the Laplacian.
In constructing the desired basis, we argue in two steps: First we deal with the spherical and then with the radial contributions of the basis functions.\\

\emph{Step 1: Poisson formula.}
By the explicit formula for the Poisson operator for the fractional Laplacian (see for instance \cite{B15}) we have that for $x \in B_1$ and some constant $c=c(n,s)>0$,  
\begin{align}
\label{eq:Poisson_form}
\begin{split}
\frac{u(x)}{c(1-|x|^2)^s} &= \int_{\R^n \setminus B_1} \frac{1}{|x-y|^n} \frac{f(y)}{(|y|^2-1)^s} \,dy = \int_1^{\infty} \int_{\partial B_1} \frac{r^{n-1}}{|x-r\omega|^n} \frac{f(r\omega)}{(r^2-1)^s} \,d\omega \,dr \\
 &= \int_1^{\infty} \int_{\partial B_1} \frac{1}{|x/r-\omega|^n} \frac{f(r\omega)}{r (r^2-1)^s} \,d\omega \,dr,
\end{split}
\end{align}
where we introduced polar coordinates $y=r\omega$ with $|\omega|=1$ in $\R^n \setminus B_1$.
Now we recall that the Poisson kernel for the standard Laplacian on the ball $B_1$ is given by $K(x,\omega) = \frac{c_n(1-|x|^2)}{|x-\omega|^n}$. We also assume that the exterior data in \eqref{eq:Poisson} have the form 
\[
f(r\omega) = r(r^2-1)^s g(r) h(\omega).
\]
Combining these observations with \eqref{eq:Poisson_form}, it follows that for some constant $\tilde{c}=\tilde{c}(n,s)>0$ 
\[
u(x) = \tilde{c} (1-|x|^2)^s \int_1^{\infty} \frac{g(r)}{1-|x|^2/r^2} \int_{\partial B_1} K(x/r, \omega) h(\omega) \,d \omega \,dr.
\]
If $h=h_{m,l}$ is a spherical harmonic of degree $m$, then $\int_{\partial B_1} K(x/r, \omega) h_{m,l}(\omega) \,d\omega = |x/r|^m h_{m,l}\left(\frac{x}{|x|}\right)$, since this is the corresponding solution of the Laplace equation. This implies that 
\begin{align}
\label{eq:u_explicit}
u(x) = \tilde{c} (1-|x|^2)^s |x|^m h_{m,l}(\omega) \int_1^{\infty} \frac{g(r)}{r^m(1-|x|^2/r^2)} \,dr.
\end{align}
The expression \eqref{eq:u_explicit} already displays the first decay properties of $u(x)$ in dependence of $m\in \N$, if we assume that $g$ is supported in $B_3 \setminus \overline{B}_2$ and is bounded. We next seek to specify $g$ in order to obtain the full decay properties, i.e.\ the decay in $m$ and $k$, where the latter will be a consequence of the choice of the radial functions $g(r)$.\\

\emph{Step 2: Construction of the radial component.}
We seek to specify the radial part of the functions $f$ in order to obtain an orthonormal basis $\{ f_{m,k,l} \}_{m,k=0}^{\infty}$ of $L^2(B_3 \setminus \overline{B}_2)$, where  
\[
f_{m,k,l}(r\omega) = r(r^2-1)^s g_{m,k}(r) h_{m,l}(\omega).
\]
Here $h_{m,l}$ are spherical harmonics of degree $m$, and the functions $g_{m,k} \in L^1([2,3])$ are to be determined. We aim for decay properties for the corresponding solutions $u_{m,k,l}$ in the form 
\begin{align}
\label{eq:decay_1}
\|u_{m,k,l}\|_{L^2(B_1)} \leq C(n,s) e^{-c(m+k)}.
\end{align}
Noting that
\begin{align}
\label{eq:decomp_u}
u_{m,k,l}(\tilde{r}\omega) = c_{n,s} (1-\tilde{r}^2)^s \tilde{r}^m h_{m,l}(\omega) F_{m,k}(\tilde{r})
\end{align}
and by invoking the orthogonality of the spherical harmonics, the decay \eqref{eq:decay_1} would follow from \eqref{eq:u_explicit}, if the function 
\[
F_{m,k}(\tilde{r}) := \int_2^{3} \frac{g_{m,k}(r)}{r^m(1-\tilde{r}^2/r^2)} \,dr
\]
satisfied $\|(1-\tilde{r}^2)^s \tilde{r}^m \tilde{r}^{(n-1)/2} F_{m,k}(\tilde{r})\|_{L^2([0,1])} \leq C(n,s)e^{-ck}$. We seek to deduce conditions which guarantee this.
To this end, we rewrite the expression for $F_{m,k}(\tilde{r})$ by spelling out the series representation of $(1-\tilde{r}^{2}/r^{2})^{-1}$:
\begin{align}
\label{eq:expand}
F_{m,k}(\tilde{r}) = \sum\limits_{j=0}^{\infty}  \tilde{r}^{2j} \int\limits_{2}^{3} \frac{g_{m,k}(r)}{r^{m+2j}} \,dr.
\end{align}
Here we used the absolute (even geometric) convergence of the sum (and the assumption that $g_{m,k}\in L^1([2,3])$) in order to exchange integration and summation.
In order to infer the desired order of vanishing, we will arrange that for $k\geq 1$ one has 
\begin{align}
\label{eq:moment_a}
\int\limits_{2}^{3} \frac{g_{m,k}(r)}{r^{m+2j}} \,dr = 0, \ 0 \leq j \leq k_0,
\end{align}
where
\begin{equation}
\label{eq:j}
\begin{split}
k_0= \frac{k-1}{2} \mbox{ (if $k$ is odd)},  \qquad
k_0 = \frac{k}{2}-1 \mbox{ (if $k$ is even)}.
\end{split}
\end{equation}
In addition, the desired orthonormality of the basis $f_{m,k,l}(r\omega)$ requires that
\begin{equation}
\label{eq:orth}
\int_2^{3} r^{n+1} (r^2-1)^{2s} g_{m,k}(r) g_{m,l}(r) \,dr = \delta_{kl}.
\end{equation}
We symmetrize the conditions \eqref{eq:moment_a} and \eqref{eq:orth} slightly, in order to deal with the same weighted scalar product in both equations.
Setting $\tilde{g}_{m,k}(r) := r^{n+1}(r^2-1)^{2s} g_{m,k}(r)$ then turns \eqref{eq:moment_a} and \eqref{eq:orth} into
\begin{equation}
\label{eq:moment_1}
\left\{ \begin{array}{c}
(\tilde{g}_{m,k}, r^{-(m+2j)})_{s} = 0 \text{ for $0 \leq j \leq k_0$ when $k \geq 1$}, \\[5pt]
(\tilde{g}_{m,k},\tilde{g}_{m,l})_{s} = \delta_{lk}
\end{array} \right.
\end{equation}
where $(f,g)_{s}:= \int\limits_{2}^{3} r^{-n-1} (r^2-1)^{-2s} f(r)g(r) \,dr$.

In order to satisfy the conditions in \eqref{eq:moment_1}, we inductively choose 
\begin{align*}
 &\tilde{g}_{m,0}(r) \in V_0 = \spa\{1\}, \\  &\tilde{g}_{m,k}(r)\in V_k:=\spa\{1,r^{-m}, r^{-m-1}, \dots, r^{-m-k+1}\} \mbox{ for } k\geq 1.
\end{align*}
To this end, we set $\tilde{g}_{m,0}(r)=c$ for some $c\in \R \setminus \{0\}$ (which is chosen such that $\tilde{g}_{m,0}$ satisfies the desired normalization $(\tilde{g}_{m,0}, \tilde{g}_{m,0})_s = 1$) and assume that for $k\geq 1$ the functions $\tilde{g}_{m,j} \in V_{j}$ are already defined for $j \leq k-1$. Then, the conditions in \eqref{eq:moment_1}, which have to be satisfied by $\tilde{g}_{m,k}$, can be formulated as
\begin{align}
\label{eq:req_1}
(\tilde{g}_{m,k},1)_s = 0, \ 
(\tilde{g}_{m,k},r^{-(m+j)})_s = 0 \mbox{ for all } 0 \leq j \leq k-2, \ (\tilde{g}_{m,k}, \tilde{g}_{m,k})_{s}=1.
\end{align}
By Gram-Schmidt orthonormalization in the finite dimensional Hilbert space $(V_k, (\cdot, \cdot)_s)$, we have that
\begin{align*}
V_k= V_{k-1} \oplus_{\perp,s} W_k,
\end{align*}
where $W_k$ is a one-dimensional vector space. Choosing $\tilde{g}_{m,k}\in W_k$ with $\|\tilde{g}_{m,k}\|_s = 1$ (up to a choice of a sign this is unique) implies the conditions from \eqref{eq:req_1}. Moreover, we note that
\begin{align*}
\left| \int\limits_{2}^3 \frac{g_{m,k}(r)}{r^{m+2j}} \,dr  \right|
&= 
\left| \int\limits_{2}^3 r^{-(n+1)}(r^2-1)^{-2s} \frac{\tilde{g}_{m,k}(r)}{r^{m+2j}} \,dr  \right|\\
&\leq \|\tilde{g}_{m,k}\|_{L^2_s([2,3])} \|r^{-\frac{n+1}{2}-m-2j} (r^2-1)^{-s}\|_{L^2([2,3])}
\leq 2^{-n/2-m-2j}.
\end{align*}
As a consequence of this, of equation \eqref{eq:expand} and of \eqref{eq:moment_a}, we obtain that
\begin{align}
\label{eq:Linfty}
|F_{m,k}(\tilde{r})| \leq 2^{-m-k}\tilde{r}^k,
\end{align}
which entails the bound
\begin{align} \label{fmk_estimate}
\|(1-\tilde{r}^2)^s \tilde{r}^m \tilde{r}^{(n-1)/2} F_{m,k}(\tilde{r})\|_{L^2([0,1])}  \leq  2^{-m-k}.
\end{align}
Together with \eqref{eq:decomp_u} this implies 
\begin{equation} \label{umkl_ltwo_estimate}
\begin{split}
\|A_0 f_{m,k,l}\|_{L^2(B_1)}&:=\|u_{m,k,l}\|_{L^2(B_1)} \\
& \leq c_{n,s}\|h_m\|_{L^2(S^{n-1})} \|(1-\tilde{r}^2)^s \tilde{r}^m \tilde{r}^{(n-1)/2} F_{m,k}(\tilde{r})\|_{L^2([0,1])} 
\leq c_{n,s} 2^{-m-k}.
\end{split}
\end{equation}
Finally, we note that for any $m\in \N$ the linear combinations of the functions $\tilde{g}_{m,k}$ generate all functions in the vector spaces $V_k$. Moreover, the set $A = \cup_{k \geq 0} V_k$ is a subalgebra of $C([2,3])$ that contains a nonzero constant function and separates points. Hence, by virtue of the Stone-Weierstrass theorem, for any $m\in \N$ finite linear combinations of the functions $\{\tilde{g}_{m,k}\}_{k\in\N}$ are dense in $C([2,3])$ and thus in particular also in $L^2_s([2,3]):=(L^2([2,3]),(\cdot,\cdot)_s)$. It follows that the functions $\{ f_{m,k,l} \}$ defined above form a dense orthonormal set in $L^2(B_3 \setminus \overline{B}_2)$. This concludes the construction of the desired basis.
\end{proof}

\begin{rmk}
We remark that the $L^{\infty}$ bound \eqref{eq:Linfty} for $F_{m,k}(r)$ could have been upgraded to an $L^{\infty}$ bound for $u_{m,k,l}$. Indeed, combining \eqref{eq:Linfty} with $L^{\infty}$ bounds for the spherical harmonics (see for instance \cite{Sogge17}) would have implied $L^{\infty}$ bounds for $u_{m,k,l}$ of the form
\begin{align*}
|u_{m,k,l}(x)| \leq C(n,s) 2^{-m-k} |x|^{m+k} |h_{m,l}(\omega)| \leq C(n,s) m^{\frac{n-2}{2}} 2^{-m-k} |x|^{m+k}.
\end{align*}
As our argument does not require these strengthened bounds, we do not pursue this further.
\end{rmk}

With the basis functions from Lemma \ref{lem:orthog} at hand, we approach our central smoothing estimate. To this end, 
we study the operator 
\begin{equation}
\label{eq:Gamma}
\Gamma(q):= \Lambda_q-\Lambda_0,
\end{equation}
and show that it has strong regularizing properties. As in \cite{M01} it is mainly in this estimate that the properties of our operator come into play. In the proof of Theorem \ref{thm:main} this smoothing estimate will be of central relevance, as it shows that the operator $q\mapsto\Gamma(q)$ ``compresses" distances strongly (c.f.\ Lemma \ref{lem:smoothing}).

\begin{prop}
\label{prop:smooth}
Let $n\geq 1$, $B_1 \subset \R^n$ and let $q \in L^{\infty}(B_1)$. Assume further that $0$ is not an eigenvalue of the operator $(-\D)^s +q$ in $B_1$ (with exterior Dirichlet data). Suppose that $f_{m,k,l}$ are as in Lemma \ref{lem:orthog}. Then there exists a universal constant $C = C_{n,s} >1$ such that for all $(m_i,k_i,l_i)$ with $m_i, k_i \in \N$, $0 \leq l_i \leq l_{m_i}$, and $i\in\{1,2\}$, one has 
\begin{align*}
(\Gamma(q)f_{m_1,k_1,l_1}, f_{m_2,k_2,l_2})_{L^2(B_3 \setminus \overline{B}_2)} 
&\leq C e^{-c\max\{m_1+k_1,m_2+k_2\}}\times\\
&\quad \times \|q\|_{L^{\infty}(B_1)}\|((-\D)^s+q)^{-1}\|_{L^2(B_1)\rightarrow L^2(B_1)}.
\end{align*}
\end{prop}

\begin{rmk}
\label{rmk:well1}
To show that $\Gamma(q)$ indeed acts on $L^2(B_3 \setminus \overline{B}_2)$, it follows from Remark \ref{rmk:well} that 
\[
\Lambda_q: L^2(B_3 \setminus \overline{B}_2) \to H^{-2s}(B_3 \setminus \overline{B}_2), \ \ f \mapsto (-\Delta)^s u|_{B_3 \setminus \overline{B}_2}
\]
is a bounded operator. The fact that $\Gamma(q) = \Lambda_q - \Lambda_0$ is bounded $L^2(B_3 \setminus \overline{B}_2) \to L^2(B_3 \setminus \overline{B}_2)$ follows since $\Gamma(q) f = (-\Delta)^s v|_{B_3 \setminus \overline{B}_2}$ where $v$ solves \eqref{eq:difference} below, and thus 
\[
\|(-\Delta)^s v\|_{L^2(B_3 \setminus \overline{B}_2)} \leq C \| v \|_{H^s(\R^n)} \leq C \| q u_0 \|_{L^2(B_1)} \leq C \|f\|_{L^2(B_3 \setminus \overline{B}_2)}
\]
by pseudolocal estimates using that $v$ is supported in $\overline{B}_1$, and by Remark \ref{rmk:well}.
\end{rmk}

\begin{rmk}
\label{rmk:support}
We stress that in contrast to the situation of the classical Calder\'on problem, in the fractional Calder\'on problem we do \emph{not} have to assume that $q$ is supported in the interior of $B_1$. This is due to the fact that in the fractional setting the smoothing property is a direct consequence of the disjointness of the two sets $B_3 \setminus \overline{B}_2$, $B_1$ (cf.\ the arguments in the proof of Lemma \ref{lem:orthog}), while in the classical case, it was necessary to use the smoothing properties of the Laplacian in order to derive regularity and decay.
\end{rmk}

\begin{proof}
Let $(m_i,k_i,l_i)$ be triplets such that $m_i, k_i \in \N$, $0 \leq l_i \leq l_{m_i}$, and $i\in\{1,2\}$. We define $u_0:= A_0 f_{m_1,k_1,l_1}$ and set $u:=A_q f_{m_1,k_1,l_1}$. Here $A_q:= r_{B_1}P_q$, where $P_q$ denotes the Poisson operator in the presence of the potential $q$, i.e., for $f\in L^2(B_3 \setminus \overline{B}_2)$, 
we set $P_q f = u$ with
\begin{align*}
(-\D)^s u + q u &= 0 \mbox{ in } B_1,\\
u &= f \mbox{ in } \R^n \setminus \overline{B}_1.
\end{align*}
Then the function $v:= u - u_0$ is a solution to
\begin{equation}
\label{eq:difference}
\begin{split}
((-\D)^s + q)v &= - q u_0  \mbox{ in } B_1 ,\\
v & = 0 \mbox{ in } \R^n \setminus \overline{B}_1.
\end{split}
\end{equation}
As a consequence,
\begin{equation}
\label{eq:smooth_a}
\begin{split}
\|v\|_{L^2(\R^n)} 
&\leq \|v\|_{L^2(B_1)} 
\leq \|((-\D)^s + q)^{-1} (q u_0)\|_{L^2(B_1)}\\
&\leq \|((-\D)^s + q)^{-1}\|_{L^2(B_1) \to L^2(B_1)} \|q\|_{L^{\infty}(B_1)} \|u_0\|_{L^2(B_1)}\\
&\leq C e^{-c(m_1+k_1)}\|((-\D)^s + q)^{-1}\|_{L^2(B_1) \to L^2(B_1)} \|q\|_{L^{\infty}(B_1)} ,
\end{split}
\end{equation}
where we used Lemma \ref{lem:orthog} to estimate $\|u_0\|_{L^2(B_1)}$.
Here the inverse operator $((-\D)^s + q)^{-1} $ is understood as the solution operator to \eqref{eq:difference} (i.e.\ to the problem with homogeneous Dirichlet data, for which boundedness properties are for instance discussed in \cite{G15}).
We next seek to estimate $(-\D)^s v|_{B_3 \setminus \overline{B}_2}$ by $\|v\|_{L^2(B_1)}$. Indeed, this follows from the disjointness of the sets $B_3 \setminus \overline{B}_2$, $B_1$: For $x\in B_3\setminus \overline{B}_2$ 
\begin{align*}
|(-\D)^s v (x)| 
&= c_{n,s}\left| \ \int\limits_{\R^n} \frac{v(x)-v(y)}{|x-y|^{n+2s}} \,dy \,\right|
= c_{n,s}\left| \ \int\limits_{B_1} \frac{v(x)-v(y)}{|x-y|^{n+2s}} \,dy \,\right|\\
&\leq c_{n,s}\int\limits_{B_1} |v(y)| \,dy
 \leq  c_{n,s}\| v \|_{L^2(B_1)}\\
&\leq  c_{n,s} e^{-c(m_1+k_1)}\|((-\D)^s + q)^{-1}\|_{L^2(B_1) \to L^2(B_1)} \|q\|_{L^{\infty}(B_1)}.
\end{align*}
Here we used that $v(x)=0$ if $x\in B_3 \setminus \overline{B}_2$ and invoked \eqref{eq:smooth_a}.
Thus,
\begin{align*}
\|(-\D)^s v\|_{L^2(B_3 \setminus \overline{B}_2)} \leq  c_{n,s} e^{-c(m_1+k_1)}\|((-\D)^s + q)^{-1}\|_{L^2(B_1) \to L^2(B_1)} \|q\|_{L^{\infty}(B_1)}.
\end{align*}
Recalling that $(-\D)^s v|_{B_3 \setminus \overline{B}_2}= \Gamma(q)f_{m_1,k_1,l_1}$ shows the statement for $m_1+k_1\geq m_2+k_2$. Using the symmetry of the Dirichlet-to-Neumann map, which hence also yields
\begin{align*}
 (\Gamma(q)f_{m_1,k_1,l_1}, f_{m_2,k_2,l_2})_{L^2(B_3 \setminus \overline{B}_2)}
=  (f_{m_1,k_1,l_1}, \Gamma(q) f_{m_2,k_2,l_2})_{L^2(B_3 \setminus \overline{B}_2)},
\end{align*} 
finally implies the result in the case that $m_1+k_1< m_2+k_2$.
\end{proof}

In the sequel, abbreviating $\mat:= (\Gamma(q)f_{m_1,k_1,l_1}, f_{m_2,k_2,l_2})_{L^{2}(B_3 \setminus \overline{B}_2)}$ and using that
\begin{align*}
\|\Gamma(q)\|_{L^2(B_3 \setminus \overline{B}_2)\rightarrow L^2(B_3 \setminus \overline{B}_2)}
&= \sup\limits_{\|f\|_{L^2(B_3 \setminus \overline{B}_2)}=1 = \|h\|_{L^2(B_3 \setminus \overline{B}_2)}}(\Gamma(q)f,h)_{L^2(B_3 \setminus \overline{B}_2)}\\
&= \sup_{ \beta, \tilde{\beta} \in \ell^2, \|\tilde{\beta}\|_{\ell^2}= \|\beta\|_{\ell^2}=1} \sum\limits_{m_i,k_i,l_i}\tilde{\beta}_{m_1,k_1,l_1} \beta_{m_2,k_2,l_2} \mat\\
&\leq \|\mat\|_{HS},
\end{align*}
where $\|\,\cdot\,\|_{HS}$ denotes the Hilbert-Schmidt norm, we identify the operator $\Gamma(q)$ with its matrix representation given by $\{\mat\}$. We note that
\begin{align*}
\|\Gamma(q)\|_{L^2(B_3 \setminus \overline{B}_2) \rightarrow L^2(B_3 \setminus \overline{B}_2)}^2 
&\leq \|\mat\|_{HS}^2
\leq \sum\limits_{m_i,k_i,l_i} |\mat|^2\\
& \leq \sup\limits_{m_i,k_i,l_i}(1+\max\{m_1+k_1,m_2+k_2\})^{2\q}|\mat|^2 \times \\
& \quad \times \sum\limits_{m_i,k_i,l_i} (1+\max\{m_1+k_1,m_2+k_2\})^{-2\q}\\
&\leq 16 \sup\limits_{m_i,k_i,l_i}(1+\max\{m_1+k_1,m_2+k_2\})^{2\q}|\mat|^2.
\end{align*}
Here we used that the number $N_p$ of tuples $(m_1,k_1,l_1,m_2,k_2,l_2)$ with $\max\{m_1+k_1,m_2+k_2\} = p$ is bounded by $8 (p+1)^{2n+1}$ (to see this, note that there are $p+1$ pairs $(m,k)$ with $m+k=p$, there are $\sum\limits_{j=0}^{p+1}(j+1) \leq (p+1)^2$ pairs $(m,k)$ with $m+k\leq p$, and the dimension of spherical harmonics of degree $\leq p$ is $\leq \sum_{m=0}^p 2(m+1)^{n-2} \leq 2(p+1)^{n-1}$). As a consequence,
\begin{align*}
\sum\limits_{m_i,k_i,l_i} (1+\max\{m_1+k_1,m_2+k_2\})^{-2\q}\leq   \sum\limits_{p=0}^{\infty} (1+p)^{-2\q} N_p\leq 16.
\end{align*}

Based on Proposition \ref{prop:smooth} and the identification of $\Gamma(q)$ with its matrix representation $\{\mat\}$, we define the following function space:

\begin{defi}
\label{defi:matrix}
We set
\begin{align*}
X&:=\left\{\{\mat\}_{m_i,k_i,l_i}: \right.\\
& \left. \quad \|\mat\|_{X}:=  \sup\limits_{m_i,k_i,l_i}(1+\max\{m_1+k_1,m_2+k_2\})^{\ql} |\mat| <\infty\right\}.
\end{align*}
\end{defi}

In the sequel, we will mainly work in this space. We remark that it defines a Banach space. Note that the above computation also yields the estimate 
\begin{equation} \label{eq:gammaq_estimate}
\|T\|_{L^2(B_3 \setminus \overline{B}_2) \rightarrow L^2(B_3 \setminus \overline{B}_2)} \leq 4 \| T \|_{X}
\end{equation}
whenever $T: L^2(B_3 \setminus \overline{B}_2) \rightarrow L^2(B_3 \setminus \overline{B}_2)$ is bounded, and the norm on the right is the $X$-norm for the matrix representation of $T$.

\begin{rmk}
\label{rmk:smooth}
As in \cite{M01} it would have been possible to consider smoother variants of the space $X$, e.g. for instance we could study the mapping properties of $\Gamma(q)$ in
\begin{align*}
X^{\gamma}
&:=\left\{\{\mat\}_{m_i,k_i,l_i}: \right.\\
& \left. \quad \|\mat\|_{X}:=  \sup\limits_{m_i,k_i,l_i}(1+\max\{m_1+k_1,m_2+k_2\})^{\ql + \gamma} |\mat| <\infty\right\}.
\end{align*}
This would correspond to investigating the Sobolev regularity of $\Gamma(q)$. As this however does not yield major new insights, we do not pursue this here.
\end{rmk}

\section{Two Metric Spaces}
\label{sec:metric}

The remainder of our argument for Theorem \ref{thm:main} and Corollary \ref{cor:main} is based on the ideas, which were introduced by Mandache \cite{M01}. Since these rely on an interplay of Proposition \ref{prop:smooth}, which is a specific property of our problem at hand, and general properties of the underlying Banach spaces we nevertheless present (most of) the arguments for completeness.

As in \cite{M01}, roughly speaking, we seek to exploit two complementary separation properties of the spaces $L^{\infty}$ and $\Gamma(L^{\infty})$. While functions in $L^{\infty}$ are ``far apart" from each other, the strong smoothing effect of $\Gamma(q)$ compresses these distances. We formalize this observation by considering $\epsilon$-discrete sets and $\delta$-nets in the corresponding function spaces. A counting argument and the pigeonhole principle then allow us to conclude the argument as in Mandache's original proof.

\begin{defi}
\label{defi:Banach}
Let $(X,d)$ be a metric space and $\epsilon>0$. A set $Y\subset X$ is an \emph{$\epsilon$-net for a set $X_1 \subset X$}, if for every $x\in X_1$ there exists an element $y\in Y$ such that $d(x,y)\leq \epsilon$. A set $Z\subset X$ is said to be \emph{$\epsilon$-discrete}, if for all $z_1,z_2 \in Z$ with $z_1 \neq z_2$ we have that $d(z_1,z_2)\geq \epsilon$.
\end{defi}

With these notions at hand, we show that the smoothing bounds from Lemma \ref{lem:smoothing} imply that the image of a suitable unit ball in $L^{\infty}$ under $\Gamma$ is ``strongly compressed":

\begin{lem}
\label{lem:smoothing}
Let $n\geq 1$, $B_1 \subset \R^n$ and $s\in(0,1)$.
Let $B^{\infty}_{r_0}$ denote the ball in $L^{\infty}(B_1)$ with radius $r_0=\lambda_{1,s}/2$, where $\lambda_{1,s}$ denotes the first Dirichlet eigenvalue of $(-\D)^s$ on $B_1$. Then, as a function of $q$, one has 
\begin{align*}
\Gamma: B^{\infty}_{r_0} \rightarrow X, \ q \mapsto \Gamma(q),
\end{align*}
where $X$ denotes the space from Definition \ref{defi:matrix}.
Moreover, there exists $C = C(n,s) >1$ such that for any $\delta \in (0,e^{-1})$ there is a $\delta$-net $Y$ for $\Gamma(B_{r_0}^{\infty})$ with
\begin{align*}
|Y| \leq e^{C|\log(\delta)|^{\qc}}.
\end{align*}
\end{lem}

\begin{rmk}
\label{rmk:radius}
We remark that the radius $r_0>0$ is chosen in such a way that for all $q \in B_{r_0}^{\infty}$ we can ensure that $0$ is not an eigenvalue of the (homogeneous) Dirichlet problem for the operator $(-\D)^s + q$.
\end{rmk}

\begin{proof}
In order to infer the mapping properties of the lemma, we invoke the estimate from Proposition \ref{prop:smooth}. As above, we identify $\Gamma(q)$ with the matrix $\{\mat\}_{m_i,k_i,l_i}$. For $q \in B^{\infty}_{r_0}$ Proposition \ref{prop:smooth} yields
\begin{align*}
|\mat|
&\leq C_{n,s} e^{-c\max\{m_1+k_1,m_2+k_2\}} \|q\|_{L^{\infty}(B_1)} \|((-\D)^s + q)^{-1}\|_{L^2(B_1) \to L^2(B_1)} \\
&\leq C_{n,s} e^{-c\max\{m_1+k_1,m_2+k_2\}} \frac{\lambda_{1,s}}{2} \|((-\D)^s + q)^{-1}\|_{L^2(B_1) \to L^2(B_1)}.
\end{align*}
It therefore suffices to estimate $\|((-\D)^s + q)^{-1}\|_{L^2(B_1) \to L^2(B_1)}$ for $q \in B_{r_0}^{\infty}$.
By our choice of $r_0$, we infer that
\[
( ((-\Delta)^s + q) v, v)_{L^2(B_1)} \geq \lambda_{1,s} \|v\|_{L^2(B_1)}^2 - r_0 \|v\|_{L^2(B_1)}^2 \geq  \frac{\lambda_{1,s}}{2} \|v\|_{L^2(B_1)}^2, \qquad v \in \widetilde{H}^s(B_1),
\]
and consequently
\begin{align*}
\|((-\D)^s + q)^{-1}\|_{L^2(B_1) \to L^2(B_1)} \leq \frac{2}{\lambda_{1,s}}.
\end{align*}
Hence,
\begin{equation} \label{eq:mat_estimate}
|\mat| \leq C_{n,s} e^{-c\max\{m_1+k_1,m_2+k_2\}}.
\end{equation}
and 
\begin{align*}
\|\mat\|_{X} 
&\leq \sup\limits_{m_i,k_i,l_i}(1+\max\{m_1+k_1,m_2+k_2\})^{\ql} C_{n,s} e^{-c\max\{m_1+k_1,m_2+k_2\}}\\
&\leq R_0 <\infty,
\end{align*}
for some sufficiently large constant $R_0>0$. This proves the desired mapping properties.

The $\delta$-net $Y$ is constructed as in \cite{M01}: For any $\delta \in (0,e^{-1})$, 
we set 
\begin{align*}
l_{\delta}:=\inf\{l_0\in \N: (1+l)^{\ql} C_{n,s} e^{-cl} \leq \delta\mbox{ for all } l\geq l_0\}, \quad  \delta':= (1+l_{\delta})^{-n-2}\delta,
\end{align*}
where $c, C_{n,s}>0$ are the constants from Proposition \ref{prop:smooth}.
In particular, these choices imply that for some constant $\tilde{C}>0$ we have $l_{\delta}\leq \tilde{C} \log(\delta^{-1})$. With this notation we further define
\begin{align*}
Y_{\delta}&:= \delta' \Z \cap [-R_0, R_0] ,\\
Y&:= \{\{\mat\}_{m_i,k_i,l_i}: \mat \in Y_{\delta} \mbox{ for } \max\{ m_1+k_1,m_2+k_2\} < l_{\delta}, \\
& \quad \quad  \mat= 0 \mbox{ else} \}.
\end{align*}
We claim that $Y$ is the desired $\delta$-net for $\Gamma(B_{r_0}^{\infty})$. To observe this, we let $a:=\{\mat\}_{m_i,k_i,l_i}$ with $a \in \Gamma(B_{r_0}^{\infty})$ and construct $b:=\{\matb\}_{m_i,k_i,l_i} \in Y$ such that $\|\matb-\mat\|_{X}\leq \delta$. Indeed, given $\{\mat\}_{m_i,k_i,l_i}\in \Gamma(B_{r_0}^{\infty})$, we define
\begin{align*}
\matb =
\left\{
\begin{array}{ll}
\argmin\limits_{\tilde{b} \in Y_{\delta}}|\tilde{b} -\mat| & \mbox{ if } \max\{m_1+k_1,m_2+k_2\} < l_{\delta},\\
0 & \mbox{ else}.
\end{array}
\right.
\end{align*}
Since $|\mat| \leq \|\mat\|_X \leq R_0$, we have 
\begin{align*}
|\mat-\matb| \leq
\left\{
\begin{array}{ll}
\delta' & \mbox{ if } \max\{m_1+k_1,m_2+k_2\} < l_{\delta},\\
|\mat| & \mbox{ else}.
\end{array}
\right.
\end{align*}
By \eqref{eq:mat_estimate} and the definition of $\delta'$ and $l_{\delta}$
\begin{align*}
\|a-b\|_{X}
= \sup\limits_{m_i,k_i,l_i}(1+\max\{m_1+k_1,m_2+k_2\})^{n+2} |\mat-\matb|\leq \delta.
\end{align*}
It remains to estimate the cardinality of $Y$. Recall from the proof of Proposition \ref{prop:smooth} that if $N_p$ is the number of tuples $(m_1,k_1,l_1,m_2,k_2,l_2)$ with $\max\{m_1+k_1,m_2+k_2\}=p$, then $N_p\leq 8(p+1)^{2n+1}$. Setting $N(\delta):= \sum\limits_{p=0}^{l_{\delta}-1} N_p$, a simple Riemann sum estimate gives the bound $N(\delta) \leq \frac{8}{2n+2} (l_{\delta}+1)^{2n+2} \leq 2 (l_{\delta}+1)^{2n+2}$. We thus have 
\begin{align*}
|Y| &\leq |Y_{\delta}|^{N(\delta)}
\leq (2R_0\delta^{-1}(1+l_{\delta})^{\q})^{2(l_{\delta}+1)^{2n+2}}
= \exp( \log(2R_0\delta^{-1}(1+l_{\delta})^{\q})2(l_{\delta}+1)^{2n+2}  ).
\end{align*}
Since $1+l_{\delta} \leq (\tilde{C}+1) \log(\delta^{-1})$, we get $|Y| \leq \exp(C \log(\delta^{-1})^{\qc})$ with $C>0$ depending on $\tilde{C}, R_0,n,s$. This concludes the argument.
\end{proof}

Complementary to the result of Lemma \ref{lem:smoothing}, which exploited the smoothing properties of $\Gamma(q)$ to show the existence of a ``relatively small" $\delta$-net, we next recall that
with respect to the $L^{\infty}$ topology, functions can be ``strongly separated". This is formalized in the presence of a ``large" $\epsilon$-discrete set.

\begin{lem}[Lemma 2 in \cite{M01}]
\label{lem:separate}
Let $n\in \N$ and $B_1 \subset \R^n$.
Let $\epsilon, \beta>0$ and for $m\in \N$ consider the set
\begin{align*}
X_{\epsilon, \beta,m} := \{f\in C_0^m(B_1): \|f\|_{L^\infty} \leq \epsilon, \ \|f\|_{C^m}\leq \beta \},
\end{align*}
with the metric induced by the $L^{\infty}$ norm. Then there exists $\mu>0$ such that for any $\beta>0$ and $\epsilon \in (0,\mu^{-1} \beta)$, there exists an $\epsilon$-discrete set $Z\subset X_{\epsilon,\beta,m}$ with 
\begin{align*}
|Z| \geq \exp(2^{-n-1}(\beta (\mu\epsilon)^{-1})^{n/m}).
\end{align*}
\end{lem}

The proof is identical as in \cite{M01}. We present it only for self-containedness.

\begin{proof}
Let $\psi \in C^{\infty}_0(B_1)$ with $\|\psi\|_{L^{\infty}} =1$. Let $\mu:= n^{m/2}\|\psi\|_{C^{m}}$. We divide the cube $[-\frac{1}{\sqrt{n}},\frac{1}{\sqrt{n}}]^n$ into $N^n$
smaller, (up to null-sets) disjoint subcubes with centers $y_1,\dots, y_N$ and side length $\frac{2}{N\sqrt{n}}$.
We consider the set
\begin{align*}
Z:=\left\{ \epsilon \sum\limits_{j=1}^{N^n} \sigma_j \psi(N \sqrt{n}(x-y_j)): \sigma_j \in \{0,1\} \right\} .
\end{align*}
Since for $\tilde{\psi}\in Z$ it holds that $\|\tilde{\psi}\|_{C^{m}}\leq \epsilon (N \sqrt{n})^m \|\psi\|_{C^m}$, we will have $\|\tilde{\psi}\|_{C^{m}} \leq \beta$ if 
\[
N := \max\{k\in \N: k \leq (\beta (\mu \epsilon)^{-1})^{1/m} \}.
\]
Since also $\|\tilde{\psi}\|_{L^{\infty}}\leq \epsilon$, we infer that $Z \subset X_{\epsilon,\beta,m}$ by our choice of $N\in \N$. The estimate on the cardinality for $Z$ follows by noting that $|Z|= 2^{N^n}$ and by plugging in the estimate $N \geq \frac{1}{2} (\frac{\beta}{\epsilon \mu})^{\frac{1}{m}}$.
\end{proof}

Combining Lemma \ref{lem:smoothing} and \ref{lem:separate} and applying the pigeonhole principle finally yields the argument for Theorem \ref{thm:main}.

\begin{proof}[Proof of Theorem \ref{thm:main}]
Let $r_0$ be as in Lemma \ref{lem:smoothing}, and let $\mu>0$ be as in Lemma \ref{lem:separate}. We choose $\beta>0$ such that $\beta/\mu = (C_1 2^{n+1})^{m/n}$, where $C_1>1$ is a sufficiently large constant that is to be determined.
Let $\epsilon\in(0,\min\{r_0/2,1, \mu^{-1}\beta\})$ and $q_0 \in L^{\infty}(B_1)$ with $\|q_0\|_{L^{\infty}}\leq r_0/2$ be arbitrary. By Lemma \ref{lem:separate} the set $q_0+ X_{m,\epsilon,\beta}$ has an $\epsilon$-discrete subset $q_0 +Z$ with
\begin{align*}
|Z| \geq \exp(2^{-n-1}(\beta (\mu\epsilon)^{-1})^{n/m}) \geq
\exp( C_1 \epsilon^{-n/m}).
\end{align*}
As $q_0 + X_{m,\epsilon,\beta} \subset B_{r_0}^{\infty}$, the set $Y$ from Lemma \ref{lem:smoothing} is a $\delta$-net for $\Gamma(q_0+X_{m,\epsilon,\beta})$.
Choosing $\delta = \exp(-\epsilon^{-\frac{n}{m} \frac{1}{2n+3}})$ thus yields that
\begin{align*}
|Y|\leq  \exp(C (\log (\delta^{-1}))^{2n+3} )\leq \exp(C \epsilon^{-n/m}).
\end{align*}
If $C_1 > C$, we thus infer that $|Z| > |Y|$ and therefore also $|q_0 + Z| > |Y|$. This implies that there exist $q_1,q_2 \in q_0+Z$ such that $\Gamma(q_1)$ and $\Gamma(q_2)$ are in the same ball of the $\delta$-net $Y$. Hence, from \eqref{eq:gammaq_estimate} we obtain that
\begin{align*}
 \|\Lambda_{q_1}-\Lambda_{q_2}\|_{L^2(B_3 \setminus \overline{B}_2) \to L^2(B_3 \setminus \overline{B}_2)} 
\leq 4\|\Gamma(q_1) - \Gamma(q_2)\|_{X} \leq 8 \delta, 
\end{align*}
which yields the first estimate in \eqref{eq:instab}. The remaining ones follow by the definition of $X_{\epsilon,\beta,m}$ and of $Z$ (in fact $\|z_1-z_2\|_{L^{\infty}} = \epsilon$ for any distinct $z_1, z_2 \in Z$).
\end{proof}

We remark that the proof for Corollary \ref{cor:main} follows along exactly the same lines as the proof of Corollary 2 in \cite{M01}, if the operator $-\D$ is replaced by the operator $(-\D)^s$. As it does not involve any new ideas, we omit its proof and refer to the argument in \cite{M01}.

\section{Approximation}
\label{sec:approx}

Again exploiting the smoothing properties of the basis from Lemma \ref{lem:orthog}, in this section we provide the argument for Theorem \ref{prop:approx}. This in particular essentially proves the optimality of the moduli of continuity in the approximation results in \cite{RS17}.

\begin{proof}[Proof of Theorem \ref{prop:approx}]
We use the basis which was constructed in Lemma \ref{lem:orthog}: We recall that the functions $f_{m,k,l}(r\omega):=g_{m,k}(r)h_{m,l}(\omega)$ from Lemma \ref{lem:orthog}, where $m, k \geq 0$ and $0 \leq l \leq l_m$, form an orthonormal basis of $L^2(B_3 \setminus \overline{B}_2)$ and are associated with the functions
$u_{m,k,l} := A_0(f_{m,k,l})$, which are solutions of \eqref{eq:Poisson}. Due to the structure of the Poisson operator for the fractional Laplacian (c.f.\ \eqref{eq:u_explicit}), these functions are again of the form 
\begin{align}
\label{eq:not}
u_{m,k,l}(r\omega) = q_{m,k}(r) h_{m,l}(\omega).
\end{align}
Here $h_{m,l}(\omega)$ denote the spherical harmonics and $q_{m,k}(r)$ are explicit functions (c.f.\ Lemma \ref{lem:orthog} and its proof), which however are in general not pairwise orthogonal.

We note that it is possible to write any function $u$ satisfying condition (i) in Theorem \ref{prop:approx} with boundary values $f= \sum\limits_{m,k,l}\alpha_{m,k,l} f_{m,k,l} \in L^2(B_3 \setminus \overline{B}_2)$ as
\begin{align}
\label{eq:decomp}
u = \sum\limits_{m,k,l} \alpha_{m,k,l} u_{m,k,l}.
\end{align}
This sum converges in $L^2(B_1)$, since by the orthogonality of the spherical harmonics $h_{m,l}$ and by \eqref{umkl_ltwo_estimate} one has 
\begin{align*}
\| \sum \alpha_{m,k,l} u_{m,k,l} \|_{L^2(B_1)}^2 &= \sum_{m,l} \| \sum_k \alpha_{m,k,l} u_{m,k,l} \|_{L^2(B_1)}^2 \leq c_{n,s}^2 \sum_{m,l} 2^{-2m} ( \sum_k |\alpha_{m,k,l}| 2^{-k})^2 \\
 &\leq c_{n,s}^2 \sum_{m,k,l} |\alpha_{m,k,l}|^2.
\end{align*}

With this preparation at hand, 
we define the sequence of functions $\{v_p\}_{p\in \N}$ that will satisfy the assertions in Theorem \ref{prop:approx} as
\begin{align*}
v_{p}(r\omega):= A_0 \left(\frac{f_{p,0,0}(r\omega)}{\|u_{p,0,0}\|_{L^2(B_1)}}\right)= \frac{u_{p,0,0}(r \omega)}{\|u_{p,0,0}\|_{L^2(B_1)}}
\end{align*}
Thus in particular, $\|v_{p}\|_{L^2(B_1)}=1$. For convenience, we also abbreviate $\alpha_0:= \|u_{p,0,0}\|_{L^2(B_1)}^{-1}$.

We now proceed in three steps:\\
\emph{Step 1:} Assuming that a function $u$ satisfies the conditions (i), (ii) in Theorem \ref{prop:approx} with a decomposition as in \eqref{eq:decomp}, we claim that
\begin{align}
\label{eq:mass}
\|\sum\limits_{k=0}^{\infty} \alpha_{p,k,0} u_{p,k,0}\|_{L^2(B_1)} \geq \frac{1}{2}.
\end{align}
\emph{Argument:} This follows from orthogonality, the normalization of $v_{p}$ and the structure of the functions $u_{m,k,l}$. Indeed, by orthonormality of the spherical harmonics and by using the notation introduced in \eqref{eq:not}, we have that
\begin{align*}
p^{-2} 
&\geq \|u - v_{p}\|_{L^2(B_1)}^2\\
& = \|\sum\limits_{k=0}^{\infty} \alpha_{p,k,0} q_{p,k} - \alpha_0 q_{p,0}\|_{L^2_r([0,1])}^2 + \sum\limits_{(m,l) \neq (p,0)}\|\sum\limits_{k=0}^{\infty}\alpha_{m,k,l} q_{p,k}\|_{L^2_r([0,1])}^2.
\end{align*}
Here $L^2_r([0,1])$ denotes the radial part of the $L^2$ norm after introducing polar coordinates and integrating out the spherical variables, and $\alpha_0= \frac{1}{\|u_{p,0,0}\|_{L^2(B_1)}}$.
In particular, 
\begin{align*}
p^{-1}
&\geq \|\sum\limits_{k=0}^{\infty} \alpha_{p,k,0} q_{p,k} - \alpha_0 q_{p,0}\|_{L^2_r([0,1])}
= \|\sum\limits_{k=0}^{\infty} \alpha_{p,k,0} u_{p,k,0} - v_{p}\|_{L^2(B_1)} \\
&\geq \|v_{p}\|_{L^2(B_1)} - \|\sum\limits_{k=0}^{\infty} \alpha_{p,k,0} u_{p,k,0}\|_{L^2(B_1)}.
\end{align*}
Choosing $p\geq 2$, recalling the normalization of $v_{p}$ and rearranging hence implies \eqref{eq:mass}.\\
\emph{Step 2:} Let $u$ be as in Step 1. We claim that there exists $k_0 \in \N \cup \{0\}$ such that 
\begin{align}
\label{eq:size}
|\alpha_{p,k_0,0}| \geq \frac{2^{p-2}}{c_{n,s}},
\end{align} 
where $c_{n,s}>0$ is the constant in \eqref{umkl_ltwo_estimate}.\\
\emph{Argument:} We argue by contradiction and assume that the claim were wrong. Thus for all $k$ we would have $|\alpha_{p,k,0}|< \frac{2^{p-2}}{c_{n,s}}$. By Step 1, the triangle inequality and the bound \eqref{umkl_ltwo_estimate} this would lead to the following chain of estimates:
\begin{align*}
\frac{1}{2} 
&\leq \|\sum\limits_{k=0}^{\infty} \alpha_{p,k,0} u_{p,k,0}\|_{L^2(B_1)} 
 \leq \sum\limits_{k=0}^{\infty}|\alpha_{p,k,0}|\| u_{p,k,0}\|_{L^2(B_1)} \\
 &< \sum\limits_{k=0}^{\infty} \frac{2^{p-2}}{c_{n,s}} c_{n,s} 2^{-p-k} = \frac{1}{2}.
\end{align*}
This is a contradiction.\\
\emph{Step 3:} Using Steps 1 and 2 we conclude the proof. Indeed, by Steps 1 and 2 for any $u=A_0(f)$ with data $f= \sum\limits_{m,k,l} \alpha_{m,k,l} f_{m,k,l} \in L^2(B_3 \setminus \overline{B}_2)$ satisfying (i), (ii) in Theorem \ref{prop:approx} for the function $v_{p}$ defined above, we on the one hand obtain that by orthonormality of the functions $f_{m,k,l}$ and with $k_0 \in \N\cup \{0\}$ as in Step 2
\begin{align*}
\|f\|_{L^2(B_3 \setminus \overline{B}_2)}
\geq |\alpha_{p,k_0,0}|\|f_{p,k_0,0}\|_{L^2(B_3 \setminus \overline{B}_2)} 
\geq \frac{2^{p-2}}{c_{n,s}}.
\end{align*}
On the other hand, by the initial normalization $\|v_{p}\|_{L^2(B_1)}=1$. Combining the last two observations and defining $c_0:=\frac{1}{4 c_{n,s}}$ implies the claim.
\end{proof}

\section{The 1D Calder\'on Problem and the Truncated Hilbert Transform}
\label{sec:1D}
In this section we discuss the one-dimensional situation in greater detail:
In the one-dimensional setting the mapping properties of $A_0$ can be made even more explicit by using a close relation between the fractional Laplacian and the truncated Hilbert transform.
To illustrate this, we again consider the Poisson operator (where for simplicity the exterior data is zero in $[-3,-2]$)
\begin{align*}
A_q:L^2([2,3]) \rightarrow L^2([-1,1]), \ g \mapsto \chi_{[-1,1]} u,
\end{align*}
where $u$ is a solution to
\begin{align*}
((-\D)^s+q) u &= 0 \mbox{ in } [-1,1],\\
u & = g \chi_{[2,3]} \mbox{ in } \R \setminus [-1,1].
\end{align*} 
By the representation formula for the solution to the fractional Laplacian in one dimension (c.f.\ for instance \cite{B15}) we infer that
\begin{align}
\label{eq:u}
\chi_{[-1,1]} A_0 g(x) = c(s)\int\limits_{[2,3]} \left( \frac{1-|x|^2}{|y|^2-1} \right)^{s} \frac{g(y)}{|x-y|} \,dy. 
\end{align}
Based on this representation formula, it will be possible to relate $A_0$ to the truncated Hilbert transform. This will allow us to transfer results on the well-studied truncated Hilbert transform to the fractional Laplacian. Before formulating this connection precisely, we recall that the Hilbert transform is defined as the singular integral operator
\begin{align*}
H:L^2(\R) \rightarrow L^2(\R), \ f\mapsto H(f)(t):=p.v.\int\limits_{\R}\frac{f(y)}{t-y} \,dy.
\end{align*}
This principal value integral is immediately seen to be well-defined on Lipschitz functions, but can be extended to the whole class of $L^2$ functions. This can for instance be achieved by noting that $H(f)=\F^{-1}(\F f(\xi) (-i \,\sgn(\xi)))$, where $\F: L^2(\R) \rightarrow L^2(\R), \ f \mapsto \F(f)= \int\limits_{\R} f(x) e^{-2\pi i x \xi} \,dx$ denotes the Fourier transform (see for example \cite{G08}, Chapter 4 for more information on singular integrals and the Hilbert transform).

We define the truncated Hilbert transforms 
\begin{align*}
 & H_{[2,3]}: L^2([2,3]) \to L^2([-1,1]), \ \ f \mapsto H(E_0 f)|_{[-1,1]}, \\
 & H_{[-1,1]}: L^2([-1,1]) \to L^2([2,3]), \ \ f \mapsto H(E_0 f)|_{[2,3]},
\end{align*}
where $E_0$ denotes extension by zero to $\R$.

\begin{lem}
\label{lem:HT_frac}
Let $A_0:L^2([2,3]) \rightarrow L^2([-1,1])$ be the mapping from above. 
Then we have that for all $x\in [-1,1]$
\begin{align*}
A_0 g (x) = -c(s)(1-x^2)^s H_{[2,3]} ((\cdot^2-1)^{-s} g)(x).
\end{align*}
\end{lem}

\begin{proof}
The proof follows immediately by rewriting the expression from \eqref{eq:u}: For $x\in [-1,1]$ we have that
\begin{align*}
A_0g(x)
&= c(s)\int\limits_{[2,3]} \left( \frac{1-|x|^2}{|y|^2-1} \right)^{s} \frac{g(y)}{|x-y|}dy\\
&= -c(s)(1-x^2)^s\int\limits_{[2,3]} \frac{(y^2-1)^{-s} g(y)}{x-y}dy\\
& = -c(s)(1-x^2)^s H_{[2,3]} ((\cdot^2-1)^{-s} g)(x). \qedhere
\end{align*}
\end{proof}

Motivated by Lemma \ref{lem:HT_frac}, we recall some properties of the singular value decomposition $(\sigma_l, f_l, g_l) \in \R_+ \times L^2([2,3]) \times L^2([-1,1])$ of the truncated Hilbert transform $H_{[2,3]}$ (see \cite{K10}). The sequences $\{f_l\}_{l\in \N}$ and $\{g_l\}_{l\in\N}$ are orthonormal bases of $L^2([2,3])$ and $L^2([-1,1])$, which are generalized eigenfunctions of the truncated Hilbert transform in the sense that
\begin{align*}
H_{[2,3]} f_l = \sigma_l g_l, \quad H_{[-1,1]} g_l = -\sigma_l f_l. 
\end{align*}
Here the completeness of the orthonormal system $\{g_l\}_{l\in \N} \subset L^2([-1,1])$ can for instance be obtained by a combination of unique continuation and duality arguments (see for instance Lemma 2.7 in \cite{Rue17}).
The singular values asymptotically satisfy exponential bounds (see \cite{KT12}):
\begin{align}
\label{eq:eigen_asymp}
0<\sigma_l \leq e^{-cl}.
\end{align}
Moreover, the functions $f_l$ are also eigenfunctions of a (degenerate) Sturm-Liouville problem. They are eigenfunctions (with weighted Neumann boundary conditions) of the operator
\begin{align}
\label{eq:L1}
L(x,d_x)f:= (P(x) f'(x))' + 2(x-\sigma)^2f(x),
\end{align}
where
\begin{align*}
P(x):=(x-1)(x+1)(x-2)(x-3), \ \sigma = \frac{5}{4}.
\end{align*}
The associated eigenvalues $\lambda_n$ satisfy $\lambda_n \geq c_1 n^2$ (see \cite{K10}, \cite{KT12}).
Based on this and Lemma \ref{lem:HT_frac}, we define function spaces and a basis adapted to the operator $A_0$:

\begin{defi}
\label{defi:function_spaces}
Let $(\sigma_l, f_l, g_l) \in \R_+ \times L^2([2,3]) \times L^2([-1,1])$ be the singular value decomposition of $H_{[2,3]}$. Then we define the weighted functions
\begin{align*}
\hat{f}_l(y):= (y^2-1)^{s} f_l(y), \quad
\hat{g}_l(x):=(1-x^2)^s g_l(x),
\end{align*}
and for $\gamma \in \R$ we consider the function spaces $L^2_{\gamma}([2,3]):= L^2([2,3])$ equipped with the scalar product
\begin{align*}
(f,g)_{L^2_{\gamma}([2,3])}:= (f,(\cdot^2-1)^{2 \gamma} g)_{L^2([2,3])}.
\end{align*}
\end{defi}

\begin{rmk}
\label{rmk:basis}
We remark that for $s\in(-1,1)$ the spaces $L^{2}([2,3])$ and $L^2_{\pm s}([2,3])$ are isomorphic. Indeed, due to the strict disjointness of the intervals $[2,3],[-1,1]$, there exist constants $c_1,c_2 \in (1/3,3)$ such that for all $s\in (-1,1)$
\begin{align*}
c_1\|f\|_{L^{2}_{s}([2,3])} \leq \|f\|_{L^2([2,3])} \leq c_2 \|f\|_{L^{2}_{s}([2,3])}.
\end{align*}
The advantage of working with the spaces $L^{2}_{\pm s}([2,3])$ instead of working with the standard space $L^2([2,3])$ is that the functions $\hat{f}_l$ are well-adapted to the fractional Laplacian. Indeed, by virtue of Lemma \ref{lem:HT_frac} they satisfy
\begin{align}
\label{eq:sing}
A_0 \hat{f}_l = -c_s (1-x^2)^s H_{[2,3]} (f_l) = -c_s (1-x^2)^s \sigma_l g_l
= -c_s \sigma_l \hat{g}_l.
\end{align}
Moreover, the functions $\{\hat{f}_l\}_{l\in \N}$ form an orthonormal basis of $L^2_{-s}([2,3])$ (as the functions $\{f_l\}_{l\in \N}$ are orthonormal with respect to $L^2([2,3])$). 
\end{rmk}

As a side product of this relation between the one-dimensional fractional Laplace operator and the truncated Hilbert transform, we derive another version of Theorem \ref{prop:approx} showing the (almost) optimality of the one-dimensional quantitative approximation result of arbitrary functions by $s$-harmonic functions (Theorem 1.3 in \cite{RS17}):

\begin{thm}
\label{thm:exp_inst_approx}
There exists a sequence $\{h_{k}\}_{k\in \N} \in L^2([-1,1])$ such that the following holds: 
For any function $\tilde{h}_{k}$ with the property that
\begin{itemize}
\item[(i)] $\tilde{h}_{k}$ approximates $h_{k}$ in the sense that
\begin{align}
\label{eq:approx_w}
\|(1-x^2)^{-s} (h_{k}-\tilde{h}_{k})\|_{L^2([-1,1])}\leq k^{-1},
\end{align}
\item[(ii)] $\tilde{h}_{k}$ satisfies the equation
\begin{align*}
(-\p_{xx})^s \tilde{h}_{k} = 0 \mbox{ in } [-1,1], \ \tilde{h}_{k} = \chi_{[2,3]} \tilde{f}_{k} \mbox{ in } \R \setminus [-1,1]
\end{align*}
for some $\tilde{f}_{k} \in L^2([2,3])$,
\end{itemize}
the associated control function $\tilde{f}_{k}$ satisfies
\begin{align*}
\|\tilde{f}_{k}\|_{L^2([2,3])}\geq C e^{c k}.
\end{align*}
\end{thm}

The proof is slightly simpler than the higher dimensional analogue in Section \ref{sec:approx}, since we can more directly exploit orthogonality.

\begin{proof}[Proof of Theorem \ref{thm:exp_inst_approx}]
We consider the functions $h_{k}:=  \hat{g}_k \in L^2([-1,1])$. Then by orthogonality of the functions $g_k$ with respect to the $L^2([-1,1])$ topology and due to the presence of the weight $(1-x^2)^{-s}$ in \eqref{eq:approx_w}, any approximation $\tilde{h}_{k}$ to $h_{k}$ within an $k^{-1}$ error threshold (in the weighted norm), has to be of the form
\begin{align*}
\tilde{h}_{k} = \alpha_k \hat{g}_k + \bar{g},
\end{align*}
with $(\bar{g},(1-x^2)^{-2s} \hat{g}_k )_{L^{2}([-1,1])}=0$ and $\alpha_k \geq 1-k^{-1}$. 
As a consequence of \eqref{eq:sing}, the associated control function $\tilde{f}_{k}$ has to be of the form 
\begin{align*}
\tilde{f}_{k} = - c_s^{-1} \sigma_k^{-1}\alpha_k \hat{f}_k + \bar{f},
\end{align*}
with $(\bar{f},\hat{f}_k )_{L^{2}_{-s}([2,3])}=0$.
Due to orthogonality in the space $L^2_{-s}([2,3])$ we therefore obtain that 
\begin{align*}
\|\tilde{f}_{k}\|_{L^2_{-s}([2,3])}\geq
C \alpha_k \sigma_k^{-1} \geq  C e^{c k}.
\end{align*} 
By virtue of the equivalence of the $L^2([2,3])$ and $L^2_{\pm s}([2,3])$ norms, this also yields a lower bound for the $L^2([2,3])$ norm of the control. 
\end{proof}

\section{Exponential Instability in Boundary-Bulk Measurements}
\label{sec:Hadamard}

Finally, we remark that the classical Hadamard example (see for instance \cite{ARV09}), showing that exponential instability can occur by passing from boundary to interior measurements, can also be directly transferred to the context of the fractional Laplacian (or rather its Caffarelli-Silvestre extension, see \cite{CS07}):

\begin{lem}
\label{lem:bound_bulk}
Let $I_s(y)$ denote the modified Bessel function of first kind.
The function $v_n(x,y)= C_s n^{-s} \sin(nx) y^{s} I_{s}(n y)$, for a suitable constant $C_s \neq 0$, is a solution to
\begin{equation}
\label{eq:exp_inst}
\begin{split}
\nabla \cdot y^{1-2s} \nabla v &= 0 \mbox{ in } \R^{2}_+,\\
v & = 0 \mbox{ on } \R \times \{0\},\\
\lim\limits_{y \rightarrow 0} y^{1-2s} \p_y v & = \sin(nx) \mbox{ on } \R \times \{0\}.
\end{split}
\end{equation}
For any $y_0>0$ this function satisfies as $n \to \infty$
\begin{align*}
v_n(x,y_0) \sim C_s n^{-s} \sin(nx) y_0^s (2\pi ny_0)^{-1/2} e^{ny_0}.
\end{align*}
\end{lem}

\begin{proof}
The result follows from a direct calculation, e.g.\ by inserting $v(x,y)$ into \eqref{eq:exp_inst}. Separating variables, the bulk equation for $I_s(ny)$ turns into a modified Bessel equation: For $z=ny$ we have
\begin{align*}
z^2 I_{s}''(z) + z I_s'(z) - (z^2+s^2) I_s(z)=0 \mbox{ in } (0,\infty). 
\end{align*}
As the modified Bessel function satisfies (see \cite{NIST})
\begin{align*}
I_{s}(z) &\sim \left( \frac{1}{2} z \right)^{s} \frac{1}{\Gamma(s+1)} \mbox{ as } z \rightarrow 0,\\
I_{s}'(z) &= I_{s+1}(z) + \frac{s}{z} I_{s}(z),
\end{align*}
the boundary conditions in \eqref{eq:exp_inst} are then satisfied by choosing the constant $C_s \neq 0$ appropriately. 
Moreover, we have (see \cite{NIST})
\begin{align*}
I_s(z) \sim \frac{e^{z}}{\sqrt{2\pi z}} \mbox{ as } z \rightarrow \infty. 
\end{align*}
The asymptotics as $n\rightarrow \infty$ thus follow from the asymptotics of $I_s(z)$ as $z \rightarrow \infty$.
\end{proof}

The example given in Lemma \ref{lem:bound_bulk} shows that there is an exponential instability in controlling bulk by boundary values for solutions to the Caffarelli-Silvestre extension problem.

\bibliographystyle{alpha}
\bibliography{citationsHTI}

\begin{thebibliography}{ARRV09}

\bibitem[Ale88]{A88}
Giovanni Alessandrini.
\newblock Stable determination of conductivity by boundary measurements.
\newblock {\em Applicable Analysis}, 27(1-3):153--172, 1988.

\bibitem[Ale90]{A90}
Giovanni Alessandrini.
\newblock Singular solutions of elliptic equations and the determination of
  conductivity by boundary measurements.
\newblock {\em Journal of Differential Equations}, 84(2):252--272, 1990.

\bibitem[ARRV09]{ARV09}
Giovanni Alessandrini, Luca Rondi, Edi Rosset, and Sergio Vessella.
\newblock The stability for the {C}auchy problem for elliptic equations.
\newblock {\em Inverse Problems}, 25(12):123004, 2009.

\bibitem[Buc16]{B15}
Claudia Bucur.
\newblock Some observations on the {G}reen function for the ball in the
  fractional {L}aplace framework.
\newblock {\em Communications in Pure and Applied Analysis}, 15(2):657--699,
  2016.

\bibitem[CS07]{CS07}
Luis Caffarelli and Luis Silvestre.
\newblock An extension problem related to the fractional {L}aplacian.
\newblock {\em Communications in partial differential equations},
  32(8):1245--1260, 2007.

\bibitem[DCR03]{DR03}
Michele Di~Cristo and Luca Rondi.
\newblock Examples of exponential instability for inverse inclusion and
  scattering problems.
\newblock {\em Inverse Problems}, 19(3):685, 2003.

\bibitem[DSV17]{DSV14}
Serena Dipierro, Ovidiu Savin, and Enrico Valdinoci.
\newblock All functions are locally $ s $-harmonic up to a small error.
\newblock {\em Journal of EMS}, 19(4):957--966, 2017.

\bibitem[Gra08]{G08}
Loukas Grafakos.
\newblock {\em Classical {F}ourier analysis}, volume~1.
\newblock Springer, 2008.

\bibitem[Gru15]{G15}
Gerd Grubb.
\newblock Fractional {L}aplacians on domains, a development of
  {H}{\"o}rmander's theory of $\mu$-transmission pseudodifferential operators.
\newblock {\em Advances in Mathematics}, 268:478--528, 2015.

\bibitem[GSU16]{GSU16}
Tuhin Ghosh, Mikko Salo, and Gunther Uhlmann.
\newblock The {C}alder{\'o}n problem for the fractional {S}chr{\"o}dinger
  equation.
\newblock {\em arXiv preprint arXiv:1609.09248}, 2016.

\bibitem[Kat10]{K10}
Alexander Katsevich.
\newblock Singular value decomposition for the truncated {H}ilbert transform.
\newblock {\em Inverse Problems}, 26(11):115011, 2010.

\bibitem[KT12]{KT12}
Alexander Katsevich and Alexander Tovbis.
\newblock Finite {H}ilbert transform with incomplete data: null-space and
  singular values.
\newblock {\em Inverse Problems}, 28(10):105006, 2012.

\bibitem[Man01]{M01}
Niculae Mandache.
\newblock Exponential instability in an inverse problem for the
  {S}chr{\"o}dinger equation.
\newblock {\em Inverse Problems}, 17(5):1435, 2001.

\bibitem[Olv10]{NIST}
Frank~WJ Olver.
\newblock {\em NIST Handbook of Mathematical Functions Hardback and CD-ROM}.
\newblock Cambridge University Press, 2010.

\bibitem[RS17]{RS17}
Angkana R{\"u}land and Mikko Salo.
\newblock The fractional {C}alder{\'o}n problem: Low regularity and stability.
\newblock {\em arXiv preprint arXiv:1708.06294}, 2017.

\bibitem[R{\"u}l17]{Rue17}
Angkana R{\"u}land.
\newblock Quantitative invertibility and approximation for the truncated
  {H}ilbert and {R}iesz transforms.
\newblock {\em arXiv preprint arXiv:1708.04285}, 2017.

\bibitem[Sog17]{Sogge17}
Christopher~D Sogge.
\newblock {\em Fourier integrals in classical analysis}, volume 210.
\newblock Cambridge University Press, 2017.

\end{thebibliography}

\end{document}